\documentclass[a4paper,10pt]{article}
\usepackage{amsmath, amssymb, graphics, amsthm}

\numberwithin{equation}{section}
\def \sgn {\operatorname{sgn}}
\def \res {\operatorname*{res}}
\def \d {\mathrm{d}}
\def \E {\mathrm{E}}
\def \C {\mathbb{C}}

\def \R {\mathbb{R}}
\def \Re {\text{Re}}

\theoremstyle{plain}
\newtheorem{proposition}{Proposition}[section]
\newtheorem{lemma}{Lemma}[section]

\theoremstyle{definition}

\theoremstyle{remark}
\newtheorem{remark}{Remark}[section]

\title{On Distribution of Product of Stable Laws}
\author{Andrea Karlov\' a \footnote{Email: andrea.karlova@gmail.com}}

\date{September 13, 2013}
\begin{document}
\maketitle
\begin{abstract}
We derive the probability distribution of product 
of two independent random variables,
each distributed according the one-dimensional stable law. 
We represent the density by its power series 
and its asymptotic expansions.
As Fox's H-functions with a particular choice of 
parameters well describe the densities of stable laws,
we discuss the choice of parameters for Fox's 
H-function such that it matches the derived densities.
As a consequence, we give representations 
of these particular Fox's H-functions
in terms of its power series.

\end{abstract}

\section{Introduction}
The integral transforms are useful tool 
for studying the distributions obtained by applying 
the binary operations on the set of independent random variables.
The applications of the integral transforms have been summarized under
certain general considerations, e.g., in \cite{zolotarev_mellin}.
In particular, Fourier transform maps the convolution operation
onto multiplication operation, and thus it is greatly used for 
studying problems of additive character.
The Mellin transform maps the multiplicative convolution onto 
the multiplication operation, see e.g. \cite{misra}, which provides  
a convenient tool for studying products of independent random 
variables. Of course, the product can be studied 
by taking the sum of logarithms of independent random variables,
thus the problem can be transformed to already well-explored additive scheme. 
The computation of the probability distribution of 
the logarithm of random variable can bring other technical difficulties.
The Mellin transform, however, provides a direct approach to study the problem
of multiplication due to its algebraic properties. 
The connection between Mellin transform 
and products of two independent 
random variables was discussed in \cite{epstein}.
The method from \cite{epstein} was generalized in \cite{thompson}
to a computation of the product of $n$ independent random variables 
and the tables for particular products of 
Cauchy and Gaussian random variables were provided.
Zolotarev \cite{zolotarev} discussed the construction of general 
multiplication scheme of independent random variables,
(referred to as M-scheme),
and the limit theorems under M-scheme were already 
explored by Abramov \cite{abramov}.

In connection with stable laws the Mellin transform and the representation 
of stable densities in terms of Fox's H-function was considered in 
\cite{schneider}. In \cite{zolotarev_mellin} the Mellin transform 
of stable laws was established in order to study 
multiplicative and divisibility properties of stable laws.
In \cite{zolotarev}, Zolotarev gives examples where the multiplication
and division theorems are applied in order to simplify 
the computation of probability distributions of statistics 
formed by the stable random variables. 
It seems that the explicit representation for distribution of 
the product of two independent stable random variables
has not been derived. 
Therefore, it is the purpose of this paper to give 
the explicit representation for the product of 
two independent stable random variables and
discuss their properties.

In this paper we firstly give an short overview on 
the Mellin transform and its application into computation of 
the product of two independent random variables. Further
we recall the elementary relations for stable random variables
and their connection to the concept of Fox's H-functions.
In the section three and four we use the inverse Mellin transform 
and Residue Theorem to establish the probability density function 
of the product of two independent stable random variables and discuss its
properties. The densities are given in the form of power series,
the connections with Fox's H-functions and concluding remarks 
are discussed in the last section.

\section{Preliminaries.}
\subsection{Product Density and Mellin Transform.}
Consider two independent random variables $X_1,X_2$,
with continuous probability density functions $f_1(x)$ and $f_2(x)$,
respectively. The probability density function $g(x)$ of their product 
$X=X_1\cdot X_2$ is given by:
\begin{align}
\label{eqn:prod}
g(x)=\int_{-\infty}^{\infty}\frac{1}{|y|}f_1(x/y)f_2(y)\mathrm{d}y
=\int_{-\infty}^{\infty}\frac{1}{|y|}f_2(x/y)f_1(y)\mathrm{d} y,
\end{align}
see e.g. \cite{epstein} for discussion.
The Mellin transform is defined for a positive random variable $X$ with 
continuous density $f(x)$ as:
\begin{align}
F(s)=\E\left[X^{s-1}\right]=\int_0^{\infty}x^{s-1}f(x)\mathrm{d} x,
\end{align}
where $s$ is a complex number, such that for some vertical strip $a_1<\Re s<a_2$,
the $F(s)$ is analytic in this strip. Therefore, the Mellin transform is defined
by a pair: a function $F(s)$ and its strip of analyticity. 
The width of the strip is determined by the behaviour of $f(x)$
for $x$ approaching the origin and for $x$ going to infinity.

The Mellin transform of product $Y=X_1\cdot X_2$ of two positive independent 
random variables $X_1,X_2$, with continuous densities $f_1(x),f_2(x)$, is given by:
\begin{align}
\label{eqn:prod1}
\nonumber
G(s)&=\E\left[Y^{s-1}\right]=\E\left[X_1^{s-1}X_2^{s-1}\right]=
\E\left[X_1^{s-1}\right]\E\left[X_2^{s-1}\right]=\\
&=F_1(s)F_2(s).
\end{align}

The inverse Mellin transform is defined as:
\begin{align}
f(x)=\frac{1}{2\pi i}\int_{a-i\infty}^{a+i\infty}x^{-s}F(s)\mathrm{d}s,
\end{align}
for all positive $x$ for which the function $f(x)$ is continuous,
where the integration path lies within the strip of analyticity of $F(s)$,
see e.g. \cite{misra} for more rigorous introduction.

Taking the inverse Mellin transform of (\ref{eqn:prod1}) yields:
\begin{align*}
g(x)=\frac{1}{2\pi i}\int_{a-i\infty}^{a+i\infty}x^{-s}F_1(s)F_2(s)\mathrm{d}s,
\end{align*}
and it can be shown that probability density computed this way is 
exactly the density in (\ref{eqn:prod}), see \cite{epstein}.

To extend this approach to a random variable $X$ which can be positive
and negative, Epstein \cite{epstein} decomposed the random variable $X$ 
with continuous density as $X=X^+-X^-$, where $X^+,X^-$, 
designates the positive and negative part
of random variable $X$, respectively.
The density is then decomposed into two continuous functions $f_1(x),f_2(x)$,
such that $f(x)=f_1(x)+f_2(x)$, where $f_1(x)=f(x)$ for $x>0$ and is $0$ 
otherwise and $f_2(x)=f(x)$ for $x<0$ and is $0$ when $x$ takes positive values.
The Mellin transform is then:
\begin{align*}
F_1(s)&=\E\left[(X^+)^{s-1}\right]=\int_0^{\infty}x^{s-1}f_1(x)\mathrm{d}x,\\
F_2(s)&=\E\left[(X^-)^{s-1}\right]=\int_0^{\infty}x^{s-1}f_2(-x)\mathrm{d}x,
\end{align*}
where $F_1(s)$ is the Mellin transform of $X$ when $x>0$ and $F_2(s)$ 
is the Mellin transform of $X$ when $x<0$.
Alternatively, Zolotarev \cite{zolotarev_mellin} treated 
this problem by introducing the truncations of random variable.
For the purpose of this article, we use the method of Epstain. 
So, applying the above described procedure on the product of two real independent 
random variables $X_1,X_2$ with continuous densities $f_1(x),f_2(x)$, 
respectively, Epstein \cite{epstein} 
decomposed the densities $f_1(x),f_2(x)$ into:
\begin{align*}
f_1(x)&=f_{11}(x)+f_{12}(x),\\
f_2(x)&=f_{21}(x)+f_{22}(x),
\end{align*}
where $f_{11}(x)=f_{21}(x)=0$ for $x$ negative and $f_{11}(x)=f_1(x)$,
$f_{21}(x)=f_2(x)$ for $x>0$, respectively. Similarly,
for $x>0$ we set $f_{12}(x)=f_{22}(x)=0$ and $f_{12}(x)=f_1(x)$,
$f_{22}(x)=f_2(x)$ for $x<0$, respectively. 
Then substituting into (\ref{eqn:prod}), we have the probability density function
$g(x)$ of the product $X=X_1\cdot X_2$ given by $g(x)=g_1(x)+g_2(x)$, where:
\begin{align*}
g_1(x)&=\int_0^{\infty}\frac{1}{y}f_{11}(x/y)f_{21}(y)\mathrm{d}y+
\int_0^{\infty}\frac{1}{y}f_{12}(-x/y)f_{22}(-y)\mathrm{d}y,\\
g_2(x)&=\int_0^{\infty}\frac{1}{y}f_{11}(x/y)f_{22}(-y)\mathrm{d}y+
\int_0^{\infty}\frac{1}{y}f_{12}(-x/y)f_{21}(y)\mathrm{d}y,
\end{align*}
and $g_1(x)>0$ for all $x>0$ and is $0$ otherwise, $g_2(x)>0$
for all negative $x$ and vanishes identically otherwise.
The Mellin transform of the product density $g(x)$ yields:
\begin{subequations}
\begin{align}
\label{eqn:prod2}
G_1(s)&=F_{11}(s)F_{21}(s)+F_{12}(s)F_{22}(s),\\
G_2(s)&=F_{11}(s)F_{22}(s)+F_{12}(s)F_{21}(s),
\end{align}
\end{subequations}
where $F_{ij}(s),\ i,j=1,2$ denotes the Mellin transform of relevant 
$f_{ij}(x),\ i,j=1,2$.
Taking the Mellin inverse of $G(s)$ gives the density of the product 
of two independent real random variables $X_1,X_2$. 
See e.g. \cite{epstein,thompson}.

\subsection{Stable Laws.}
The stable laws plays a central role in the problems considering convergence 
of the sums of independent random variables. Let us briefly recall its definition. 
Consider mutually independent random variables $X,X_1,X_2,\ldots$,
with probability distribution function $F(x)$ and denote the sum of $n$ elements 
as $S_n=X_1+\ldots+X_n$. The distribution $F(x)$ is stable if for each 
$n$ there exists constants $a_n>0$ and $b_n\in\R$, such that 
$S_n\overset{d}{=}a_nX+b_n$; i.e. the probability distribution function 
of the sum $S_n$ differs only by scale and, or, respectivelly, location 
from the distribution function $F(x)$; see e.g. \cite{feller2}.  
If $b_n=0$ for all $n$, the stable distribution is called 
strictly stable. In this paper, we confine our attention only to
strictly stable distributions.

In general, the densities of strictly stable distributions belong to the 
class of special functions, see \cite{zolotarev_specf}. 
The characteristic function $\varphi(k)$ of a strictly stable law is given by:
\begin{align}
\label{eqn:c}
\log \varphi(k)=\left\{
\begin{array}{c}
C k^{\alpha},\quad\text{for }k>0,\\
C^{\ast} |k|^{\alpha},\quad\text{for }k<0,
\end{array} \right .
\end{align}
where $C\in\C$, and where $C^{\ast}$ designates the complex conjugate 
of complex number $C$. We require $\Re~C<0$ in order to keep $\varphi(k)$ 
bounded. Parameter $\alpha$ is called the stability parameter of the stable law
and takes values: $0<\alpha\leq2$. The Gaussian probability law is a special
case of stable law for $\alpha=2$. The Cauchy distribution corresponds
to strictly stable law with $\alpha=1$. 
Depending on the form of complex number $C$ in (\ref{eqn:c}), many different 
parametrizations have been proposed, see e.g. \cite{zolotarev}. In this paper 
we use the parametrization discussed in \cite{karlova}.
We consider the geometric form of complex number $C$ and we define:
\begin{align}
\log \varphi(k;\alpha,p_1,c)=
-c|k|^{\alpha}\exp\left\{-i\alpha(p_1-p_2)\tfrac{\pi}{2} \sgn k\right\},
\end{align}
where $c>0$ and $p_1+p_2=1$. The condition $\Re~C<0$ implies the boundaries:
$0\leq p_1,p_2\leq 1$ for $0<\alpha<1$ and 
$1-\tfrac{1}{\alpha}\leq p_1,p_2\leq \tfrac{1}{\alpha}$ for $1<\alpha\leq 2$.
The parameter $c$ is a scaling parameter of stable law. The parameter $p_1$ is 
equal to the integral of the stable density over the positive real line, i.e.
$1-p_1$ is the value of the probability distribution function at $x=0$.
Therefore the parameter $p_1$ represents the asymmetry of stable law. 
Obviously, when $p_1=0.5$, the stable law is symmetric. For any admissible 
$p_1>0.5$, the stable law is skewed to the right, whereas for any admissible
$p_1<0.5$, the stable law is skewed to the left. The bounds on the parameters 
$p_1,p_2$ indicate the admissible proportion of probability mass on the positive
and negative parts of the real line; i.e. how maximally asymmetric the distribution
can be for different values of $\alpha$. 

The conditions on the parameters $p_1,p_2$ guarantee the integrability 
of $|\varphi(k;\alpha,p_1,c)|$ over the whole real line, and so it is valid 
to represent the stable density by the inverse Fourier integral as:
\begin{align}
\label{eqn:std}
f(x;\alpha,p_1,c)=
\frac{1}{\pi}\Re\int_0^{\infty}e^{-ikx}\varphi(k;\alpha,p_1,c)\mathrm{d}k.
\end{align}
It can be easily seen from (\ref{eqn:std}) that for any real $x$ the following two
properties hold:
\begin{subequations}
\begin{align}
\label{eqn:std1}
f(x;\alpha,p_1,c)&=f(-x;\alpha,p_2,c),\\
\label{eqn:std2}
f(x;\alpha,p_1,c)&=\tfrac{1}{c^{1/\alpha}}f(xc^{-1/\alpha};\alpha,p_1,1).
\end{align}
\end{subequations}
We refer to the first property as the reflection property 
of stable density, and to the second property as 
the scaling property of stable density. 
Note, that if $c=1$, we omit it in the notation.

Power series of stable densities were firstly completely determined by 
\cite{bergstrom}. In our notation, they are given by the series:
\begin{subequations}
\begin{align}
\label{eqn:std3}
f(x;\alpha,p_1)=\frac{1}{\pi}\sum_{k=1}^{\infty}
\frac{\Gamma(k/\alpha+1)}{\Gamma(k+1)}\sin (p_2 k\pi)x^k.
\end{align}
These series are absolutely and uniformly convergent in every finite domain 
on the real line for $1<\alpha\leq2$. For $\alpha=2$, the sum of the power series
equals the density of the normal distribution with variance $2$.
Although the series converges rapidly for any $|x|\leq 1$, 
it converges rather slowly for large values of $x$, 
as it takes many terms before the gamma functions can dominate the power $x^k$.
For $\alpha=1$ and $p_1=0.5$ the series reduces to geometric series and converges
only inside the unit circle. For $\alpha<1$, the series is divergent, however,
it is its asymptotic expansion for $|x|\ll 1$; see \cite{bergstrom}.

The asymptotic expansions for the stable densities for $1<\alpha<2$  
were also derived in \cite{bergstrom}, and are of the form:
\begin{align}
\label{eqn:std4}
f(x;\alpha,p_1)\sim\frac{1}{\pi}\sum_{k=1}^{\infty}
\frac{\Gamma(\alpha k+1)}{\Gamma(k+1)}\sin (\alpha p_2 k\pi)x^{-\alpha k-1},
\end{align}
\end{subequations}
for $x\gg 1$.
The series is divergent for $1<\alpha<2$. For $0<\alpha<1$ the series converges
and is a power series of density of that stable law; see \cite{bergstrom,zolotarev}.
For $\alpha=1$, the series converges outside the unit circle.

The Mellin transform of the density of stable law $f(x;\alpha,p_1)$ 
was given in \cite{zolotarev_mellin,schneider}. In our notation, it is:
\begin{subequations}
\begin{align}
\label{eqn:mellin}
F(s;\alpha,p_1)&=
\frac{1}{\alpha\pi}\Gamma(s)\Gamma\left(\tfrac{1-s}{\alpha}\right)
\sin(p_1[1-s]\pi)=\\
\nonumber
&=\frac{\Gamma(s)\Gamma\left(\tfrac{1-s}{\alpha}\right)}
{\alpha\Gamma(p_1-p_1s)\Gamma(p_2+p_1s)},
\end{align}
for $x>0$, and the strip of analyticity is: $0<\Re~s<\alpha+1$. 
The reflection property (\ref{eqn:std1}) implies
that for $x<0$ Mellin transform of $f(x;\alpha,p_1)$ equals 
$F(s;\alpha,p_2)$. 
The scaling property of stable law (\ref{eqn:std2}) implies that for
the Mellin transform of density $f(x;\alpha,p_1,c)$ we have:
\begin{align}
\label{eqn:mellin_scale}
F(s;\alpha,p_1,c)=c^{(s-1)/\alpha}F(s;\alpha,p_1,1),
\end{align}
\end{subequations}
see e.g. \cite{zolotarev_mellin,misra}. The probability density of stable law 
is represented by Mellin inverse in the form of Mellin-Barnes integral,
see e.g. \cite{bateman1}, as:
\begin{align}
\label{eqn:mellin1}
f(x;\alpha,p_1)=\frac{1}{2\alpha\pi i}\int_{a-i\infty}^{a+i\infty}
\frac{\Gamma(s)\Gamma\left(1-\tfrac{\alpha-1}{\alpha}-\tfrac{s}{\alpha}\right)}
{\Gamma(1-p_2-p_1 s)\Gamma(p_2+p_1s)}x^{-s}\mathrm{d}s,
\end{align}
for $x>0$. 
For $x<0$ the probability density of stable law is computed by using 
(\ref{eqn:std1}). 
Relation (\ref{eqn:mellin1}) indicates 
that the density of stable law 
can be written in terms of Fox's H-function
\footnote{Fox's H-function was introduced in \cite{fox}. It is defined 
by the Mellin-Barnes integral as:
\begin{align*}
H^{mn}_{pq}
\left[z\Big|
\begin{array}{ccc}
(a_1,\alpha_1),\ldots,(a_p,\alpha_p)\\
(b_1,\beta_1),\ldots,(b_q,\beta_q)
\end{array}
\right]
=\frac{1}{2\pi i}\int_{\gamma}g(s)z^{-s}\mathrm{d}s,
\end{align*}
where:
\begin{align*}
g(s)=\frac{\prod_{j=1}^m\Gamma(b_j+\beta_j s)\prod_{j=1}^n\Gamma(1-a_j-\alpha_j s)}
{\prod_{j=m+1}^{q}\Gamma(1-b_j-\beta_js)\prod_{j=n+1}^p\Gamma(a_j+\alpha_j s)},
\end{align*}
and where the integration path $\gamma$ is a vertical line in complex plane
indented to avoid the poles of the integrand $g(s)$.} 
as:
\begin{align}
f(x;\alpha,p_1)=\tfrac{1}{\alpha}H^{11}_{22}\left[x\Big|
\begin{array}{cc}
(\tfrac{\alpha-1}{\alpha},\tfrac{1}{\alpha}),(p_2,p_1)\\
(0,1),(p_2,p_1)
\end{array}\right],
\end{align}
for $x>0$, see \cite{schneider}.

The density of stable law given by Mellin-Barnes integral in (\ref{eqn:mellin1})
can be evaluated using the Residue Theorem. 
By carrying that computation it can be seen that the terms in 
the power series (\ref{eqn:std3}) 
corresponds to residues of Mellin transform of density of stable law, and 
these residues lie on the left-hand side of integration path in 
(\ref{eqn:mellin1}). Similarly, the terms in the series (\ref{eqn:std4}) 
are residues of Mellin transform of density of stable law, 
and these residues lie on the right-hand side of integration path 
in (\ref{eqn:mellin1}), see \cite{schneider}.

\section{Product of two independent stable random variables.}
In this section we establish the density of product 
of two independent random variables $X_1,X_2$, with stable probability densities 
$f_1(x;\alpha_1,p_1)$ and $f_2(x;\alpha_2,q_1)$, respectively, in terms of their
power series. The Mellin transforms $F_1(s;\alpha_1,p_1)$, $F_2(s;\alpha_2,q_2)$ 
of the stable densities $f_1(x;\alpha_1,p_1)$ and $f_2(x;\alpha_2,q_1)$, respectively,
are given in (\ref{eqn:mellin}) for $x>0$. 
From Mellin transform of product density $g(x)$ in (\ref{eqn:prod2}) and reflection
property of stable density (\ref{eqn:std1}) we have:
\begin{subequations}
\begin{align}
\label{eqn:G1}
G_1(s;\alpha_1,\alpha_2,p_1,q_1)=F_1(s;\alpha_1,p_1)F_2(s;\alpha_2,q_1)+
F_1(s;\alpha_1,p_2)F_2(s;\alpha_2,q_2)
\end{align} 
for $x>0$, and 
\begin{align}
\label{eqn:G2}
G_2(s;\alpha_1,\alpha_2,p_1,q_1)=F_1(s;\alpha_1,p_1)F_2(s;\alpha_2,q_2)+
F_1(s;\alpha_1,p_2)F_2(s;\alpha_2,q_1)
\end{align}
for $x<0$. The product density $g(x)$ is given by:
\begin{align}
g(x;\alpha_1,\alpha_2,p_1,q_1)&=
\frac{1}{2\pi i}\int_{a-i\infty}^{a+i\infty}x^{-s}G_1(s)\mathrm{d}s,\\
g(-x;\alpha_1,\alpha_2,p_1,q_1)&=
\frac{1}{2\pi i}\int_{a-i\infty}^{a+i\infty}x^{-s}G_2(s)\mathrm{d}s,
\end{align}
where $x>0$, and for simplicity, $0<a<1$. 

In the evaluation of the latter Mellin-Barnes integrals, clearly
\begin{align}
\label{eqn:st_mb}
\tilde{g}(x;\alpha_1,\alpha_2,p_1,q_1)=\frac{1}{2\pi i}\int_{a-i\infty}^{a+i\infty}
F_1(s;\alpha_1,p_1)F_2(s;\alpha_2,q_1)x^{-s}\mathrm{d}s,
\end{align}
\end{subequations}
is of particular interest. 
The power series of $\tilde{g}(x)$ 
is established in Lemmas \ref{thm:lem1}, \ref{thm:lem2}. 

The scaling property (\ref{eqn:std2}) and (\ref{eqn:mellin_scale})
implies the scaling property of product density:
\begin{subequations}
\begin{align}
\label{eqn:stprod_scale}
g(x;\alpha_1,\alpha_2,p_1,q_1,c_1,c_2)&=\tfrac{1}{
c_1^{1/\alpha_1}c_2^{1/\alpha_2}}
g\left(xc_1^{-1/\alpha_1}c_2^{-1/\alpha_2};\alpha_1,\alpha_2,p_2,q_2\right).
\end{align}

For the simplest case when random variables $X_1,X_2$ 
are identically distributed independent random variables, having symmetric stable 
density, i.e. $f_1(x)=f_2(x)\equiv f(x;\alpha,1/2)$, we have: 
$G_1(s)=2\left[F_1(s;\alpha,1/2)\right]^2$, and $G_2(s)=G_1(s)$. 
Then the product density $g(x)$ simply equals $2\tilde{g}(x)$ for real $x$.

In the general case, close examination of (\ref{eqn:G1}) and 
(\ref{eqn:G2}) implies that:
\begin{align}
\label{eqn:stprod_refl1}
g(x;\alpha_1,\alpha_2,p_1,q_1)&=g(x;\alpha_1,\alpha_2,p_2,q_2),\\
\label{eqn:stprod_refl2}
g(x;\alpha_1,\alpha_2,p_1,q_2)&=g(-x;\alpha_1,\alpha_2,p_1,q_1).
\end{align}
\end{subequations}
This implies the reflection property of product density:
\begin{align}
g(x;\alpha_1,\alpha_2,p_1,q_1)=g(-x;\alpha_1,\alpha_2,p_1,q_2)=
g(-x;\alpha_1,\alpha_2,p_2,q_1).
\end{align}
In the following Proposition we establish the power-series representation
of $g(x)$.
\begin{proposition}
\label{thm:prop1}
Consider two independent random variables $X_1,X_2$, with strictly stable
probability densities $f_1(x;\alpha_1,p_1)$ and $f_2(x;\alpha_2,q_1)$, 
respectively. Assume that the stability parameters $\alpha_1,\alpha_2$ 
are either $1<\alpha_1,\alpha_2\leq 2$ or satisfy condition that $2/3<\alpha_1<1$ 
and $\alpha_2>\tfrac{\alpha_1}{2\alpha_1-1}$. The asymmetry parameters $p_1,q_1$
can have any admissible values except 0. 

Then for $x\neq0$, the density $g(x)$ of the random variable $X=X_1\cdot X_2$
has the power series:
\begin{subequations}
\begin{align}
\label{eqn:st_prod1}
g(x;\alpha_1,\alpha_2,p_1,q_1)&=\tfrac{2}{\pi^2}\sum_{k=1}^{\infty}
\Gamma\left(\tfrac{k}{\alpha_1}+1\right)
\Gamma\left(\tfrac{k}{\alpha_2}+1\right)
\sin(p_1k\pi)\sin(q_1k\pi)\xi(k)
\frac{x^{k-1}}{(k!)^2}\\
\nonumber
&+\tfrac{p_2-q_1}{\pi}\sum_{k=1}^{\infty}
\Gamma\left(\tfrac{k}{\alpha_1}+1\right)
\Gamma\left(\tfrac{k}{\alpha_2}+1\right)\sin([p_1+q_1]k\pi)
\frac{x^{k-1}}{(k!)^2}\\
\nonumber
&+\tfrac{p_1-q_1}{\pi}\sum_{k=1}^{\infty}
\Gamma\left(\tfrac{k}{\alpha_1}+1\right)
\Gamma\left(\tfrac{k}{\alpha_2}+1\right)\sin([p_1-q_1]k\pi)
\frac{x^{k-1}}{(k!)^2},
\end{align}
where $\xi(k)$ designates the auxiliary function given by :
\begin{align}
\xi(k)\equiv2\psi(k)-\tfrac{1}{\alpha_1}\psi\left(\tfrac{k}{\alpha_1}\right)-
\tfrac{1}{\alpha_2}\psi\left(\tfrac{k}{\alpha_2}\right)-\log |x|,
\end{align}
and $\psi(k)$ is a digamma function.
\end{subequations}
\end{proposition}

From (\ref{eqn:st_prod1}) we see that the commutative property 
of multiplication is satisfied.

For some specific choices of stability and asymmetry parameters of stable law,
the sum given in (\ref{eqn:st_prod1}) simplifies. 
Consider the special case of $\alpha_1=\alpha_2=\alpha$,
with $1<\alpha\leq2$. Let us consider the case when one stable law 
is symmetric, say, e.g. $q_1=1/2$,
and the second stable density is asymmetric, i.e. $p_1\neq1/2$. 
Then the power series (\ref{eqn:st_prod1}) becomes:
\begin{align*}
g(x;\alpha,p_1,1/2)&=\tfrac{2}{\alpha^2\pi^2}\sum_{k=0}^{\infty}(-1)^k
\left[\Gamma\left(\tfrac{2k+1}{\alpha}\right)\right]^2
\sin(p_1[2k+1]\pi)\xi(2k+1)
\frac{x^{2k}}{[(2k)!]^2}-\\
&-\tfrac{p_1-p_2}{\alpha^2\pi}\sum_{k=0}^{\infty}(-1)^k
\left[\Gamma\left(\tfrac{2k+1}{\alpha}\right)\right]^2
\cos(p_1[2k+1]\pi)\frac{x^{2k}}{[(2k)!]^2}.
\end{align*}
It is easy to verify that the latter series is a power-series of a symmetric
function. Therefore, we conclude that by multiplying a stable random variable 
with a symmetric stable random variable, we obtain symmetric law.

The following Proposition is a consequence of Proposition \ref{thm:prop1}
and summarizes the problem when $X_1,X_2$ are independent and identically 
distributed stable random variables with the density $f(x;\alpha,p_1)$ 
and $1<\alpha\leq2$. 

\begin{proposition}
\label{thm:prop2}
Consider random variable $X=X_1\cdot X_2$. The random variables $X_1,X_2$ 
are independent and identically distributed, with strictly stable density 
$f(x;\alpha,p_1)$. It is assumed that $1<\alpha\leq 2$, and the asymmetry 
parameter $p_1$ has any admissible value. 

For $x\neq0$ the density $g(x)$ of random variable $X$ 
is represented by the power series:
\begin{subequations}
\begin{align}
\label{eqn:st_prod2a}
g(x;&\alpha,p_1)=\tfrac{2}{\pi^2}\sum_{k=1}^{\infty}
\left[\Gamma\left(\tfrac{k}{\alpha}+1\right)\right]^2
[\sin(p_1k\pi)]^2\tilde{\xi}(k)\frac{x^{k-1}}{(k!)^2}
\end{align}
for the asymmetric strictly stable law, i.e. $p_1\neq 1/2$.
For symmetric stable law, i.e. when $p_1=1/2$, we have:
\begin{align}
\label{eqn:st_prod2b}
g(x;\alpha,1/2)&=\tfrac{2}{\alpha^2\pi^2}\sum_{k=0}^{\infty}
\left[\Gamma\left(\tfrac{2k+1}{\alpha}\right)\right]^2\xi(2k+1)
\frac{x^{2k}}{(2k!)^2}.
\end{align}
The auxiliary functions $\tilde{\xi}(k)$ and $\xi(k)$ are defined
as follows:
\begin{align}
\xi(k)&=2\psi(k)-\tfrac{2}{\alpha}\psi\left(\tfrac{k}{\alpha}\right)-\log |x|,\\
\tilde{\xi}(k)&=\tfrac{1}{2}\xi(k)+(p_1-p_2)\psi(p_1 k)-(p_1-p_2)\psi(1-p_1 k).
\end{align}
\end{subequations}
\end{proposition}

The series (\ref{eqn:st_prod2b}) further simplifies for $\alpha=2$. 
Using Duplication formulas for the gamma function 6.1.18 in 
\cite{abramowitz}, and the digamma function 6.3.8 in \cite{abramowitz}, 
respectively, we get:
\begin{align*}
g(x;2)&=\tfrac{1}{\pi}\sum_{k=0}^{\infty}
\frac{1}{(k!)^2}\Big\{\psi(k+1)-\log\tfrac{x}{4}\Big\}
\left(\tfrac{x}{4}\right)^{2k}.
\end{align*}
It can be verified that this equals $\tfrac{1}{\pi}K_0(x/2)$.
This is the density of the product of two independent normal random variables,
see e.g.\cite{epstein}.

The power series representation for density of product of 
two stable random variables given by (\ref{eqn:st_prod1})-(\ref{eqn:st_prod2b}) 
assumed mainly $1<\alpha\leq 2$. The following proposition covers the case 
for $0<\alpha<1$. 
\begin{proposition}
\label{thm:prop3}
Consider two independent random variables $X_1,X_2$, with strictly stable
probability densities $f_1(x;\alpha_1,p_1)$ and $f_2(x;\alpha_2,q_1)$, 
respectively. Assume that for the stability parameters $\alpha_1,\alpha_2$ holds
that either $0<\alpha_1,\alpha_2<1$ or stability
parameters $\alpha_1,\alpha_2$ satisfy the condition: $\alpha_1>2/3$
and $\alpha_2<\tfrac{\alpha_1}{2\alpha_1-1}$. 
The asymmetry parameters $p_1,q_1$ can have any admissible values
except 0. 

Then for $x>0$, the density $g(x)$ of random variable $X=X_1\cdot X_2$ is
given by power series:
\begin{subequations}
\begin{align}
\label{eqn:st_prod3a}
g(x;&\alpha_1,\alpha_2,p_1,q_1)=\\
\nonumber
&=\tfrac{1}{\alpha_1\pi^2}\sum_{k=1}^{\infty}
\tfrac{(-1)^{k+1}}{(k-1)!}
\left[\Gamma(\alpha_1 k+1)\right]^2
\Gamma\left(1-\tfrac{\alpha_1}{\alpha_2}k\right)
\phi(k;\alpha_1,p_1,q_1)x^{-\alpha_1k-1}\\
\nonumber
&+\tfrac{1}{\alpha_2\pi^2}\sum_{k=1}^{\infty}
\tfrac{(-1)^{k+1}}{(k-1)!}
\left[\Gamma(\alpha_2 k+1)\right]^2
\Gamma\left(1-\tfrac{\alpha_2}{\alpha_1}k\right)
\phi(k;\alpha_2,p_1,q_1)x^{-\alpha_2k-1}
\end{align}
for $\alpha_1\neq\alpha_2$, where we have defined:
\begin{align}
\phi(k;\alpha,p_1,q_1)\equiv\cos(\alpha[p_1-q_1]k\pi)-
\cos(\alpha k\pi)\cos(\alpha [p_1-q_2]k\pi)
\end{align}
For $\alpha_1=\alpha_2=\alpha$ we have the power series:
\begin{align}
\label{eqn:st_prod3b}
g(x;\alpha,p_1,q_1)&=\tfrac{1}{\pi^2}\sum_{k=1}^{\infty}
\left[\Gamma\left(\alpha k+1\right)\right]^2
\phi(k;\alpha,p_1,q_1)\zeta(k)
\frac{x^{-\alpha k-1}}{(k!)^2}\\
\nonumber
&-\tfrac{p_1+q_1}{2\pi}\sum_{k=1}^{\infty}
\left[\Gamma\left(\alpha k+1\right)\right]^2
\sin(\alpha[p_1+q_1] k \pi)\frac{x^{-\alpha k-1}}{(k!)^2}\\
\nonumber
&-\tfrac{p_2+q_2}{2\pi}\sum_{k=1}^{\infty}
\left[\Gamma\left(\alpha k+1\right)\right]^2
\sin(\alpha[p_2+q_2] k \pi)\frac{x^{-\alpha k-1}}{(k!)^2}\\
\nonumber
&+\tfrac{p_1-q_1}{\pi}\sum_{k=1}^{\infty}
\left[\Gamma\left(\alpha k+1\right)\right]^2
\sin(\alpha [p_1-q_1] k \pi)\frac{x^{-\alpha k-1}}{(k!)^2},
\end{align}
where:
\begin{align}
\label{eqn:zeta}
\zeta(k)\equiv \tfrac{2}{\alpha}\psi(k+1)-2\psi\left(\alpha k+1\right)
+\log |x|.
\end{align}
\end{subequations}
\end{proposition}
For computation of values for $x<0$, we use reflection property 
(\ref{eqn:stprod_refl2}).

\begin{remark}
When $X_1,X_2$ are independent identically distributed stable random
variables with density $f(x;\alpha,p_1)$ and $0<\alpha<1$, 
the density of the product $X_1\cdot X_2$ reduces to:
\begin{subequations}
\begin{align}
\label{eqn:st_prod3c}
g(x;\alpha,p_1)&=
\tfrac{2}{\pi^2}\sum_{k=1}^{\infty}
\left[\Gamma\left(\alpha k+1\right)\right]^2
[\sin(\alpha p_1k\pi)]^2\tilde{\zeta}(k,p_1)\frac{x^{-\alpha k-1}}{(k!)^2}+\\
\nonumber
&+\tfrac{2}{\pi^2}\sum_{k=1}^{\infty}
\left[\Gamma\left(\alpha k+1\right)\right]^2
[\sin(\alpha p_2k\pi)]^2\tilde{\zeta}(k,p_2)\frac{x^{-\alpha k-1}}{(k!)^2},
\end{align}
where 
\begin{align}
\tilde{\zeta}(k,p_1)&=\tfrac{1}{2}\zeta(k)+p_1\psi(\alpha p_1 k)
-p_1\psi(1-\alpha p_1 k),
\end{align}
\end{subequations}
and $\zeta(k)$ is defined in (\ref{eqn:zeta}).
\end{remark}

\section{Proofs.}
In this section we give proofs to Propositions 
\ref{thm:prop1}-\ref{thm:prop3} established in the previous section. 
Let us start with Lemma, which gives power series of function $\tilde{g}(x)$ 
defined in (\ref{eqn:st_mb}).  
\begin{lemma}
\label{thm:lem1}
For $x>0$ the function $\tilde{g}(x)$ defined by Mellin-Barnes integral as:
\begin{subequations}
\begin{align}
\label{eqn:st_mb1}
\frac{1}{2\pi i}\int_{a-i\infty}^{a+i\infty}
\left[\Gamma(s)\right]^2
\Gamma\left(\tfrac{1-s}{\alpha_1}\right)
\Gamma\left(\tfrac{1-s}{\alpha_2}\right)
\sin(p_1[1-s]\pi)\sin(q_1[1-s]\pi)x^{-s} 
\tfrac{\mathrm{d}s}{\alpha_1\alpha_2\pi^2},
\end{align}
where $0<a<1$, $\alpha_1+\alpha_2<2\alpha_1\alpha_2$,
and $p_1,q_1\neq 0$,
has the power series representation as follows:
\begin{align}
\label{eqn:st_prod11}
\tilde{g}(x)&=\tfrac{1}{\pi^2}\sum_{k=1}^{\infty}
\Gamma\left(\tfrac{k}{\alpha_1}+1\right)
\Gamma\left(\tfrac{k}{\alpha_2}+1\right)
\sin(p_1k\pi)\sin(q_1k\pi)\xi(k)
\frac{x^{k-1}}{(k!)^2}\\
\nonumber
&-\tfrac{p_1+q_1}{2\pi}\sum_{k=1}^{\infty}
\Gamma\left(\tfrac{k}{\alpha_1}+1\right)
\Gamma\left(\tfrac{k}{\alpha_2}+1\right)\sin([p_1+q_1]k\pi)
\frac{x^{k-1}}{(k!)^2}\\
\nonumber
&+\tfrac{p_1-q_1}{2\pi}\sum_{k=1}^{\infty}
\Gamma\left(\tfrac{k}{\alpha_1}+1\right)
\Gamma\left(\tfrac{k}{\alpha_2}+1\right)\sin([p_1-q_1]k\pi)
\frac{x^{k-1}}{(k!)^2},
\end{align}
where $\xi(k)$ designates the auxiliary function given by :
\begin{align}
\xi(k)\equiv2\psi(k)-\tfrac{1}{\alpha_1}\psi\left(\tfrac{k}{\alpha_1}\right)-
\tfrac{1}{\alpha_2}\psi\left(\tfrac{k}{\alpha_2}\right)-\log x,
\end{align}
and $\psi(k)$ is digamma function.
\end{subequations}
\end{lemma}

\begin{proof}[Proof of Lemma \ref{thm:lem1}:]
We use Residue Theorem for the evaluation of the integral 
(\ref{eqn:st_mb1}). 
Consider the closed loop $\gamma$ defined as $\gamma=\gamma_1+\gamma_2$, 
where $\gamma_1(t)=a+t, t\in[-ir,ir]$ and 
$\gamma_2(t)=a+re^{it}, t\in[\pi/2,3\pi/2]$, $r>0$. The loop $\gamma$ is formed
by half-circle on the left-side of integration path of (\ref{eqn:st_mb1}) and
matches the part of line of integration path. We denote the integrand in 
(\ref{eqn:st_mb1}) as 
\begin{align}
\label{eqn:h1}
h(s)\equiv\left[\Gamma(s)\right]^2
\Gamma\left(\tfrac{1-s}{\alpha_1}\right)
\Gamma\left(\tfrac{1-s}{\alpha_2}\right)
\sin(p_1[1-s]\pi)\sin(q_1[1-s]\pi)x^{-s},
\end{align}
and integrate $h(s)$ along the loop $\gamma$.
On the left-hand side of integration path in (\ref{eqn:st_mb1}) 
function $h(s)$ has poles of second order at all non-positive integer values. 
To evaluate the residue of $h(s)$ at $s=-k$, we rewrite $(s+k)^2 h(s)$ as:
\begin{align*}
h_1(s)=\frac{[\Gamma(s+k+1)]^2}{(s+k-1)^2\ldots s^2}
\Gamma\left(\tfrac{1-s}{\alpha_1}\right)
\Gamma\left(\tfrac{1-s}{\alpha_2}\right)
\sin(p_1[1-s]\pi)\sin(q_1[1-s]\pi)x^{-s},
\end{align*}
and we expand each term in $h_1(s)$ on the neighbourhood of point $s=-k$ 
as follows:
\begin{align*}
(s+n)^{-2}=\tfrac{1}{(k-n)^2}+\tfrac{2}{(k-n)^3}(s+k)+\ldots,\quad
\text{for }0\leq n\leq k-1,
\end{align*}
and:
\begin{align*}
&[\Gamma(s+k+1)]^2=1+2\Gamma'(1)(s+k)+\ldots\\
&\Gamma\left(\tfrac{1-s}{\alpha_1}\right)=\Gamma\left(\tfrac{1+k}{\alpha_1}\right)-
\tfrac{1}{\alpha_1}\Gamma'\left(\tfrac{1+k}{\alpha_1}\right)(s+k)+\ldots,\\
&x^{-s}=x^k\left\{1-(s+k)\log x+\ldots\right\},\\
&\sin(p_1[1-s]\pi)=\sin(p_1[1+k]\pi)-p_1\pi\cos(p_1[1+k]\pi)(s+k)+\ldots.
\end{align*}
Here we have used the Binomial Theorem, see e.g. 3.6.9 \cite{abramowitz}, 
and Taylor expansion, see e.g. 3.6.4 \cite{abramowitz}. 
In order to express the expansion for $h_1(s)$, 
we multiply the above series together,
and look for the coefficient associated with the first order term.
It is given by:
\begin{align}
\label{eqn:res1}
\frac{x^k}{(k!)^2}&\Gamma\left(\tfrac{k+1}{\alpha_1}\right)
\Gamma\left(\tfrac{k+1}{\alpha_2}\right)
\sin(p_1[1+k]\pi)\sin(q_1[1+k]\pi)
\Big\{2\Gamma'(1)+\\
\nonumber
&+\sum_{n=1}^{k}\tfrac{2}{n}-
\tfrac{1}{\alpha_1}\psi\left(\tfrac{k+1}{\alpha_1}\right)-
\tfrac{1}{\alpha_2}\psi\left(\tfrac{k+1}{\alpha_2}\right)-\log x\Big\}\\
\nonumber
&-\pi\frac{x^k}{(k!)^2}\Gamma\left(\tfrac{k+1}{\alpha_1}\right)
\Gamma\left(\tfrac{k+1}{\alpha_2}\right)
\Big\{p_1\cos(p_1[1+k]\pi)\sin(q_1[1+k]\pi)\\
\nonumber
&+q_1\sin(p_1[1+k]\pi)\cos(q_1[1+k]\pi)\Big\},
\end{align}
where $\psi(s)=\Gamma'(s)/\Gamma(s)$ is digamma function; 
see e.g. \cite{bateman1}.
Using relation 6.3.2 in \cite{abramowitz} for the digamma function $\psi(s)$, 
and 4.3.33 in \cite{abramowitz} for circular functions,
we obtain the value of the residue of $h(s)$ at $s=-k$. That is:
\begin{align*}
\res_{s=-k}h(s)&=
\frac{x^k}{(k!)^2}\Gamma\left(\tfrac{k+1}{\alpha_1}\right)
\Gamma\left(\tfrac{k+1}{\alpha_2}\right)\sin(p_1[1+k]\pi)\sin(q_1[1+k]\pi)
\Big\{2\psi(k+1)-\\
&-\tfrac{1}{\alpha_1}\psi\left(\tfrac{k+1}{\alpha_1}\right)-
\tfrac{1}{\alpha_2}\psi\left(\tfrac{k+1}{\alpha_2}\right)-\log x\Big\}-\\
&-\tfrac{\pi}{2}\frac{x^k}{(k!)^2}\Gamma\left(\tfrac{k+1}{\alpha_1}\right)
\Gamma\left(\tfrac{k+1}{\alpha_2}\right)
\Big\{(p_1+q_1)\sin\left[(p_1+q_1)(1+k)\pi\right]+\\
&-(p_1-q_1)\sin\left[(p_1-q_1)(1+k)\pi\right]\Big\}.
\end{align*}

Let $m$ denote the biggest integer smaller then $r$. The orientation of 
loop $\gamma$ is anti-clockwise, so the index function has value $1$. 
The Residue Theorem gives:
\begin{align*}
\int_{a-ir}^{a+ir}h(t)\mathrm{d}t=2\pi i\sum_{k=0}^m\res_{s=-k}h(s)
-\int_{\gamma_2}h(s)\mathrm{d}s.
\end{align*}
Let us consider the convergence of integral:
\begin{align}
\label{eqn:int_arc}
\int_{\gamma_2}h(s)\mathrm{d}s=\int_{\pi/2}^{3\pi/2}h\left(re^{it}\right)ire^{it}\mathrm{d}t.
\end{align}
We use the asymptotic formula for the gamma function 6.1.39 in \cite{abramowitz},
and the Stirling formula for $s=u+iv=re^{it}$
to estimate the upper bound
of factors of $h(s)$. 
This yields: $|\Gamma(s)|\leq\sqrt{2\pi}r^{u-1/2}e^{u-vt}$.
Regarding the integral (\ref{eqn:int_arc}), 
the real part of $s$ is negative, and so we estimate:
\begin{align*}
&|\Gamma(s)|^2\leq2\pi r^{2u-1}e^{-2u-2vt}\leq c_1 r^{-2r-1},\\
&\left|\Gamma\left(\tfrac{1-u}{\alpha_1}-i\tfrac{v}{\alpha_1}\right)\right|\leq
\sqrt{2\pi}r^{-u/\alpha_1-(1/2-1/\alpha_1)}e^{u/\alpha_1+vt/\alpha_1}
e^{-1/\alpha_1}\leq c_2 r^{r/\alpha_1},\\
&|\sin(p_1[1-s]\pi)|\leq c_3 \cosh(-p_1\pi vt),\\
&|x^{-s}|\leq c_4 x^r,
\end{align*}
where $c_i,i=1,2,3,4$, are some real constants. So, we estimate:
\begin{align*}
\left|\int_{\gamma_2}h(s)\mathrm{d}s\right| \leq\pi r \max_{s\in\gamma_2}|h(s)|
\leq c r^{-r(2-1/\alpha_1-1/\alpha_2)},
\end{align*}
where again $c$ is some real constant. 
In order for the integral (\ref{eqn:int_arc}) to vanish identically
as $r$ goes to infinity, we require:
$\alpha_1+\alpha_2<2\alpha_1\alpha_2.$
Therefore, for any admissible $\alpha_1,\alpha_2$, we apply Jordan's Lemma 
on integral (\ref{eqn:int_arc}). 
Thus we have:
\begin{align*}
\tfrac{1}{2\pi i}\int_{a-i\infty}^{a+i\infty}h(s)\mathrm{d}s=\sum_{k=0}^{\infty}
\res_{s=-k}h(s).
\end{align*}
After substituting back into (\ref{eqn:st_mb1}) and rearranging the terms 
in the series we obtain (\ref{eqn:st_prod11}).
\end{proof}

\begin{proof}[Proof of Proposition \ref{thm:prop1}:]
We first discuss the condition on the stability parameters 
$\alpha_1,\alpha_2$. To apply results of previous Lemma, the parameters 
must satisfy $\alpha_1+\alpha_2<2\alpha_1\alpha_2$.
This inequality is satisfied for $1<\alpha_1,\alpha_2\leq 2$
or for pairs $\alpha_1,\alpha_2$ satisfying condition: 
$\alpha_1>2/3$ and $\tfrac{\alpha_1}{2\alpha_1-1}<\alpha_2$.

From Lemma \ref{thm:lem1} and (\ref{eqn:G1}), 
we compute the series: 
\begin{subequations}
\begin{align}
\label{eqn:g11}
g_1(x)=\tilde{g}(x;\alpha_1,\alpha_2,p_1,q_1)+
\tilde{g}(x;\alpha_1,\alpha_2,p_2,q_2),
\end{align}
where $\tilde{g}(x)$ is defined in (\ref{eqn:st_mb}).
From $p_1+p_2=1$ and addition formulas for circular functions follows:
\begin{align*}
\sin(p_2 k\pi)&=(-1)^{k+1}\sin(p_1 k\pi),\\
\sin([p_2+q_2] k\pi)&=-\sin([p_1+q_1]k\pi),\\
\sin([p_2-q_2] k\pi)&=-\sin([p_1-q_1]k\pi).
\end{align*}
Therefore from (\ref{eqn:st_prod11}) we obtain:
\begin{align}
\label{eqn:tg1}
\tilde{g}(x;\alpha_1,\alpha_2,p_2,q_2)&=
\tilde{g}(x;\alpha_1,\alpha_2,p_1,q_1)+\\
\nonumber
&+\tfrac{1}{\pi}\sum_{k=1}^{\infty}
\Gamma\left(\tfrac{k}{\alpha_1}+1\right)
\Gamma\left(\tfrac{k}{\alpha_2}+1\right)\sin([p_1+q_1]k\pi)
\frac{x^{k-1}}{(k!)^2},
\end{align}
\end{subequations}
and we proved the relation (\ref{eqn:st_prod1}).

\begin{subequations}
Similarly, we have:
\begin{align*}
g_2(x)=\tilde{g}(x;\alpha_1,\alpha_2,p_1,q_2)+
\tilde{g}(x;\alpha_1,\alpha_2,p_2,q_1),
\end{align*}
and we have relations:
\begin{align*}
\sin([p_1+q_2] k\pi)&=-\sin([p_2+q_1]k\pi)=(-1)^k\sin([p_1-q_1]k\pi),\\
\sin([p_1-q_2] k\pi)&=-\sin([p_2-q_1]k\pi)=(-1)^k\sin([p_1+q_1]k\pi),
\end{align*}
and so:
\begin{align*}
\tilde{g}(x;\alpha_1,\alpha_2,&p_1,q_2)=
\tilde{g}(-x;\alpha_1,\alpha_2,p_1,q_1)+\\
&+\tfrac{1}{2\pi}\sum_{k=1}^{\infty}
\Gamma\left(\tfrac{k}{\alpha_1}+1\right)
\Gamma\left(\tfrac{k}{\alpha_2}+1\right)\sin([p_1+q_1]k\pi)
\frac{(-x)^{k-1}}{(k!)^2}+\\
&+\tfrac{1}{2\pi}\sum_{k=1}^{\infty}
\Gamma\left(\tfrac{k}{\alpha_1}+1\right)
\Gamma\left(\tfrac{k}{\alpha_2}+1\right)\sin([p_1-q_1]k\pi)
\frac{(-x)^{k-1}}{(k!)^2},
\end{align*}
and
\begin{align*}
\tilde{g}(x;\alpha_1,\alpha_2,&p_2,q_1)=
\tilde{g}(-x;\alpha_1,\alpha_2,p_1,q_1)+\\
&+\tfrac{1}{2\pi}\sum_{k=1}^{\infty}
\Gamma\left(\tfrac{k}{\alpha_1}+1\right)
\Gamma\left(\tfrac{k}{\alpha_2}+1\right)\sin([p_1+q_1]k\pi)
\frac{(-x)^{k-1}}{(k!)^2}-\\
&-\tfrac{1}{2\pi}\sum_{k=1}^{\infty}
\Gamma\left(\tfrac{k}{\alpha_1}+1\right)
\Gamma\left(\tfrac{k}{\alpha_2}+1\right)\sin([p_1-q_1]k\pi)
\frac{(-x)^{k-1}}{(k!)^2}.
\end{align*}
\end{subequations}
Thus, the same power series gives us both $g_1$ and $g_2$, as 
$g_2(x)=g_1(-x)$. This should have been expected since
the power series for $g_1(x)$ converges uniformly
in every bounded region of the complex x plane,
therefore defines an entire function.
\end{proof}

\begin{proof}[Proof of Proposition \ref{thm:prop2}:]	
Let us assume that $x>0$. First we consider $\tilde{g}(x;\alpha,p_1)$.
When $p_1\neq1/2$, we rewrite (\ref{eqn:res1}) to obtain residue:
\begin{align*}
\res_{s=-k}h(s)=
\frac{x^k}{(k!)^2}\left[\Gamma\left(\tfrac{k+1}{\alpha}\right)\right]^2
&[\sin(p_1[1+k]\pi)]^2\Big\{2\psi(k+1)-\\
&-\tfrac{2}{\alpha}\psi\left(\tfrac{k+1}{\alpha}\right)-
2p_1\pi\cot(p_1[1+k]\pi)-\log x\Big\},
\end{align*}
for $\tilde{g}(x;\alpha,p_2)$ and $p_2\neq 1/2$, we have:
\begin{align*}
\res_{s=-k}h(s)=
\frac{x^k}{(k!)^2}\left[\Gamma\left(\tfrac{k+1}{\alpha}\right)\right]^2
&[\sin(p_1[1+k]\pi)]^2\Big\{2\psi(k+1)-\\
&-\tfrac{2}{\alpha}\psi\left(\tfrac{k+1}{\alpha}\right)+
2(1-p_1)\pi\cot(p_1[1+k]\pi)-\log x\Big\},
\end{align*}
and so for the term in $g(x;\alpha,p_1)$ we get:
\begin{align*}
2\frac{x^k}{(k!)^2}&\left[\Gamma\left(\tfrac{k+1}{\alpha}\right)\right]^2
[\sin(p_1[1+k]\pi)]^2\Big\{\psi(k+1)-\\
&-\tfrac{1}{\alpha}\psi\left(\tfrac{k+1}{\alpha}\right)
-(p_1-p_2)\pi\cot(p_1[1+k]\pi)-\tfrac{1}{2}\log x\Big\}.
\end{align*}
Using reflection formula for digamma function 6.3.7 \cite{abramowitz}, 
we obtain for the term in $g(s;\alpha,p_1)$ the following expression:
\begin{align*}
2\frac{x^k}{(k!)^2}&\left[\Gamma\left(\tfrac{k+1}{\alpha}\right)\right]^2
[\sin(p_1[1+k]\pi)]^2\Big\{\psi(k+1)-\tfrac{1}{\alpha}
\psi\left(\tfrac{k+1}{\alpha}\right)-\\
&-(p_1-p_2)\psi(1-p_1[1+k])+(p_1-p_2)\psi(p_1[1+k])-
\tfrac{1}{2}\log x\Big\}.
\end{align*}
Thus, we have shown (\ref{eqn:st_prod2a}).
For $p_1=1/2$, the sum (\ref{eqn:st_prod2b}) follows directly from 
(\ref{eqn:st_prod1}).
\end{proof}

\begin{lemma}
\label{thm:lem2}
For $x>0$ the function $\tilde{g}(x)$ defined by Mellin-Barnes integral 
in (\ref{eqn:st_mb1}), 
with parameters satisfying $\alpha_1+\alpha_2>2\alpha_1\alpha_2$
and $p_1,q_1\neq 0$,
has the power series representation as follows:
\begin{subequations}
\begin{align}
\label{eqn:st_prod21}
\tilde{g}(x)&=
\tfrac{1}{\alpha_1\pi^2}\sum_{k=1}^{\infty}
\tfrac{(-1)^{k+1}}{(k-1)!}
\left[\Gamma(\alpha_1 k+1)\right]^2
\Gamma\left(1-\tfrac{\alpha_1}{\alpha_2}k\right)
\sin(\alpha_1 p_1 k\pi)\sin(\alpha_1 q_1 k \pi) x^{-\alpha_1 k-1}+\\
\nonumber
&+\tfrac{1}{\alpha_2\pi^2}\sum_{k=1}^{\infty}
\tfrac{(-1)^{k+1}}{(k-1)!}
\left[\Gamma(\alpha_2 k+1)\right]^2
\Gamma\left(1-\tfrac{\alpha_2}{\alpha_1}k\right)
\sin(\alpha_2 p_1 k\pi)\sin(\alpha_2 q_1 k \pi) x^{-\alpha_2 k-1},
\end{align}
for $\alpha_1\neq\alpha_2$. When $\alpha_1=\alpha_2=\alpha$
the power series is given by:
\begin{align}
\label{eqn:st_prod22}
\tilde{g}(x)&=\tfrac{1}{\pi^2}\sum_{k=0}^{\infty}
\left[\Gamma\left(\alpha k+1\right)\right]^2
\sin(\alpha p_1 k\pi)\sin(\alpha q_1 k \pi)\zeta(k)
\frac{x^{-\alpha k-1}}{(k!)^2}-\\
\nonumber
&-\tfrac{p_1+q_1}{2\pi}\sum_{k=0}^{\infty}
\left[\Gamma\left(\alpha k+1\right)\right]^2
\sin(\alpha [p_1+q_1]k \pi)\frac{x^{-\alpha k-1}}{(k!)^2}+\\
\nonumber
&+\tfrac{p_1-q_1}{2\pi}\sum_{k=0}^{\infty}
\left[\Gamma\left(\alpha k+1\right)\right]^2
\sin(\alpha[p_1-q_1]k \pi)\frac{x^{-\alpha k-1}}{(k!)^2}.
\end{align}
and we define:
\begin{align}
\zeta(k)\equiv 
\tfrac{2}{\alpha}\psi(k+1)-2\psi\left(\alpha k+1\right)+\log x.
\end{align}
\end{subequations}
\end{lemma}

\begin{proof}[Proof of Lemma \ref{thm:lem2}:]
We proceed as in Proof of Lemma 4.1, only this time 
we consider the residues which lie on the right-side of integration path
and use Residue Theorem to evaluate the integral (\ref{eqn:st_mb1}).
So, take the closed loop $\gamma$ defined as $\gamma=\gamma_1-\gamma_2$, 
where $\gamma_1(t)=a+t, t\in[-ir,ir]$ and 
$\gamma_2(t)=a+re^{it}, t\in[-\pi/2,\pi/2]$, $r>0$. The loop $\gamma$ is formed
by the half-circle on the right-side of integration path of (\ref{eqn:st_mb1}) and
coincides with the part of integration path. 
We denote the integrand in (\ref{eqn:st_mb1}) as $h(s)$
and integrate it along the loop $\gamma$.
Whenever $\alpha_1\neq\alpha_2$,
function $h(s)$ has only simple poles on the right-side 
of integration path given in (\ref{eqn:st_mb1}). These are located 
at all points $\alpha_1 k+1$ and $\alpha_2 k+1$, where $k$ is a positive 
integer. When $\alpha_1=\alpha_2=\alpha$, 
the function $h(s)$ has second order poles at all points $\alpha k+1$,
and $k$ being a positive integer.

The evaluation of the simple poles is straightforward, and yields:
\begin{align*}
\res_{s=\alpha_1 k+1}h(s)=\alpha_1\tfrac{(-1)^{k+1}}{k!}
\left[\Gamma(1+\alpha_1 k)\right]^2 
\Gamma\left(-\tfrac{\alpha_1}{\alpha_2}k\right)
\sin(\alpha_1 p_1 k\pi)\sin(\alpha_1 q_1 k \pi)x^{-\alpha_1 k-1}.
\end{align*}
Using the functional equation for Gamma function 6.1.17 \cite{abramowitz}, 
we have:
\begin{align*}
\res_{s=\alpha_1 k+1}h(s)=\alpha_2
\tfrac{(-1)^{k}}{(k-1)!}
\left[\Gamma(\alpha_1 k+1)\right]^2
\Gamma\left(1-\tfrac{\alpha_1}{\alpha_2} k\right)
\sin(\alpha_1 p_1 k\pi)\sin(\alpha_1 q_1 k \pi)
x^{-\alpha_1 k-1}.
\end{align*}
The residues of the simple poles at $\alpha_2 k+1$ with $k$ being 
positive integer are computed similarly. 

For $\alpha_1=\alpha_2=\alpha$, 
we rewrite $h(s)$ in (\ref{eqn:h1}) as:
\begin{align*}
h(s)=\frac{\left[\Gamma\left(\tfrac{1-s}{\alpha}+k\right)\right]^2}
{\left(\tfrac{1-s}{\alpha}+k-1\right)^2\ldots\left(\tfrac{1-s}{\alpha}\right)^2}
\left[\Gamma(s)\right]^2\sin(p_1[1-s]\pi)\sin(q_1[1-s]\pi)x^{-s},
\end{align*}
and expand the factors in $(s-\alpha k-1)^2 h(s)$ around $\alpha k+1$:
\begin{align*}
&\left(\tfrac{1-s}{\alpha}+n\right)^{-2}=
\tfrac{1}{(k-n)^2}-\tfrac{2}{\alpha (k-n)^3}(s-\alpha k-1)+
\ldots,\quad\text{ for }0\leq n\leq k-1,
\end{align*}
and:
\begin{align*}
&\left[\Gamma\left(\tfrac{1-s}{\alpha}+k+1\right)\right]^2=
1-\tfrac{2}{\alpha}\Gamma'(1)(s-\alpha k-1)+\ldots\\
&\left[\Gamma\left(s\right)\right]^2=\left[\Gamma\left(\alpha k+1\right)\right]^2
\left\{1+2\psi\left(\alpha k+1\right)(s-\alpha k-1)+\ldots\right\},\\
&x^{-s}=x^{-\alpha k-1}\left\{1-(s-\alpha k-1)\log x+\ldots\right\},\\
&\sin(p_1[1-s]\pi)=-\sin(\alpha p_1 k\pi)-p_1\pi\cos(\alpha p_1 k\pi)(s-\alpha k-1)
+\ldots.
\end{align*}
The residue of the second order pole of $h(s)$ is now given by 
the coefficient associated with the first order term of the 
expansion of $h(s)$, and so:
\begin{align*}
\res_{s=\alpha k+1}h(s)&=
-\alpha^2\frac{x^{-\alpha k-1}}{(k!)^2}
\left[\Gamma\left(\alpha k+1\right)\right]^2
\sin(\alpha p_1 k\pi)\sin(\alpha q_1 k \pi)
\Big\{\tfrac{2}{\alpha}\psi(k+1)-\\
&-2\psi\left(\alpha k+1\right)+\log x\Big\}+\\
&+\tfrac{\alpha^2\pi}{2}\frac{x^{-\alpha k-1}}{(k!)^2}
\left[\Gamma\left(\alpha k+1\right)\right]^2
\Big\{(p_1+q_1)\sin(\alpha [p_1+q_1] k \pi)+\\
&-(p_1-q_1)\sin(\alpha [p_1-q_1] k\pi)\Big\} .
\end{align*}
Let $m$ denotes the biggest integer smaller then $r$. The orientation of 
loop $\gamma$ is clockwise, so the index function has value $-1$. 
The Residue Theorem gives:
\begin{align*}
\int_{a-ir}^{a+ir}h(t)\mathrm{d}t=\int_{\gamma_2}h(s)\mathrm{d}s
-2\pi i\sum_{k=1}^m\res_{s=\alpha_1 k+1}h(s)
-2\pi i\sum_{k=1}^m\res_{s=\alpha_2 k+1}h(s),
\end{align*}
for $\alpha_1\neq\alpha_2$ and 
\begin{align*}
\int_{a-ir}^{a+ir}h(t)\mathrm{d}t=\int_{\gamma_2}h(s)\mathrm{d}s
-2\pi i\sum_{k=1}^m\res_{s=\alpha k+1}h(s),
\end{align*}
when $\alpha_1=\alpha_2=\alpha$.
To verify the Jordan's Lemma for integral along arc $\gamma_2$,
we use the same estimates as in Proof of Lemma \ref{thm:lem1}, 
only this time the real part of $s$ is positive. 
So, denoting $s=u+iv=re^{it}$, we estimate $u\leq r$ and we have:
\begin{align*}
\left|\int_{\gamma_2}h(s)\mathrm{d}s\right| \leq\pi r \max_{s\in\gamma_2}|h(s)|
\leq c r^{r(2-1/\alpha_1-1/\alpha_2)},
\end{align*}
where $c$ is some real constant. 
The integral along arc $\gamma_2$ vanishes identically
as the modulus $r$ goes to infinity, 
whenever $\alpha_1+\alpha_2>2\alpha_1\alpha_2$.
This inequality is obviously satisfied for $0<\alpha_1,\alpha_2<1$, 
or for pairs $\alpha_1,\alpha_2$ satisfying condition: 
$\alpha_1>2/3$ and $\tfrac{\alpha_1}{2\alpha_1-1}>\alpha_2$.
In particular, when $\alpha_1=\alpha_2=\alpha$, we have $0<\alpha<1$.
Therefore, for any admissible $\alpha_1,\alpha_2$, we apply Jordan's Lemma 
on the integral of $h(s)$ integrated along the arc $\gamma_2$. 
So after substituting back into 
(\ref{eqn:st_mb1}) we obtain (\ref{eqn:st_prod21}),
(\ref{eqn:st_prod22}), respectively.
\end{proof}

\begin{proof}[Proof of Proposition \ref{thm:prop3}:]
Let us start with the series (\ref{eqn:st_prod3a}), i.e. we assume that 
$\alpha_1\neq\alpha_2$. We need to compute (\ref{eqn:g11}).
In order to prove the relation, we use formulas for circular functions 
4.3.31, 4.3.32 in \cite{abramowitz} and obtain:
\begin{align*}
\sin(\alpha_1p_1k\pi)&\sin(\alpha_1q_1k\pi)+
\sin(\alpha_1p_2k\pi)\sin(\alpha_1q_2k\pi)=\\
&=\cos(\alpha_1[p_1-q_1]k\pi)-\tfrac{1}{2}\cos(\alpha_1[p_1+q_1]k\pi)
-\tfrac{1}{2}\cos(\alpha_1[p_2+q_2]k\pi)=\\
&=\cos(\alpha_1[p_1-q_1]k\pi)-
\cos(\alpha_1k\pi)\cos(\alpha_1[p_1-q_2]k\pi),
\end{align*}
and from Lemma \ref{thm:lem2} follows (\ref{eqn:st_prod3a}).

Relation (\ref{eqn:st_prod3b}) is established the same way as 
(\ref{eqn:st_prod1})  in Proposition \ref{thm:prop1}.
\end{proof}

\section{Discussion.}
We investigated the product of two independent random variables which 
are distributed according the strictly stable law. 
In particular, when we consider independent random variables $X_1,\ldots,X_4$, 
with stable densities $f_1(x)=f_2(x)\equiv f(x;\alpha,p_1)$ 
and $f_3(x)=f_4(x)\equiv f(x;\alpha,p_2)$, respectively, 
as a consequence of reflection and scaling properties (\ref{eqn:stprod_scale}) -
(\ref{eqn:stprod_refl2}) we have:
\begin{subequations}
\begin{align}
X_1\cdot X_2\overset{d}{=}X_3\cdot X_4\overset{d}{=}-(X_1\cdot X_3).
\end{align}
The density of the product, say $X_1\cdot X_2$, 
is given by (\ref{eqn:st_prod2a}), (\ref{eqn:st_prod2b}), for $1<\alpha\leq2$,
and by (\ref{eqn:st_prod3c}) for $0<\alpha<1$, respectively. 
If $f_1(x)=f_2(x)\equiv f(x;\alpha,p_1,c)$ 
and $f_3(x)=f_4(x)\equiv f(x;\alpha,p_2,c)$ are densities of random variables
$X_1,\ldots,X_4$, then for its density (\ref{eqn:st_prod2a}) holds the scaling
property:
\begin{align}
g(x;\alpha,p_1,c)=
\tfrac{1}{c^{2/\alpha}}g(xc^{-2/\alpha};\alpha,p_1)=
\tfrac{1}{c^{2/\alpha}}g(xc^{-2/\alpha};\alpha,p_2).
\end{align}
\end{subequations}
Therefore, the scaling property of stable laws is preserved under the 
multiplicative operation. 
Further, if of the stable laws in the considered product is symmetric, 
then the resulting product density is also symmetric. 

The representations of densities given in Proposition \ref{thm:prop3}, 
(\ref{eqn:st_prod3a})-(\ref{eqn:st_prod3c}) are asymptotic expansions 
of $g(x;\alpha_1,\alpha_2,p_1,q_1)$ when $1<\alpha_1,\alpha_2<2$ 
and $x\gg1$. Similarly the series in Propositions \ref{thm:prop1}, 
\ref{thm:prop2} are asymptotic expansions for product density 
$g(x;\alpha_1,\alpha_2,p_1,q_1)$ when $0<\alpha<1$ and $x\ll 1$. 

In \cite{zolotarev_mellin} are discussed multiplicative laws for stable 
distributions for particular choices of parameters. We have not considered
this problem. We also have not given the value of the product density at origin.
The distribution of the ratio of independent stable random variables 
can be derived the same way, see \cite{epstein}.

The Lemmas \ref{thm:lem1}, \ref{thm:lem2}, suggest that we can describe 
the power series we derived for the product of stable law variables 
in terms of Fox's H-function.
The representation of stable densities by Fox's H-functions  
has been considered in \cite{schneider} for some specific choices 
of asymmetry parameters. Let us examine the particular choices 
of parameters for Fox's H-function in connection with the density 
of product of two independent stable laws.
Recalling (\ref{eqn:mellin1}), we represent function $\tilde{g}(x)$ 
in (\ref{eqn:st_mb}) as:
\begin{align*}
&\tilde{g}(x;\alpha_1,\alpha_2,p_1,q_1)=\\
&=\frac{1}{2\alpha_1\alpha_2\pi i}\int_{a-i\infty}^{a+i\infty}
\frac{[\Gamma(s)]^2\Gamma\left(1-\tfrac{\alpha_1-1}{\alpha_1}-\tfrac{s}{\alpha_1}
\right)\Gamma\left(1-\tfrac{\alpha_2-1}{\alpha_2}-\tfrac{s}{\alpha_2}\right)}
{\Gamma(1-p_2-p_1 s)\Gamma(p_2+p_1s)\Gamma(1-q_2-q_1 s)\Gamma(q_2+q_1s)}
x^{-s}\d s,
\end{align*}
for $x>0$. This can be written in terms of Fox's H-function
as:
\begin{align*}
\tilde{g}(x;\alpha_1,\alpha_2,p_1,q_1)=
\tfrac{1}{\alpha_1\alpha_2}H^{22}_{44}\left[x\Big|
\begin{array}{cccc}
(\tfrac{\alpha_1-1}{\alpha_1},\tfrac{1}{\alpha_1}),(\tfrac{\alpha_2-1}{\alpha_2},\tfrac{1}{\alpha_2}),(p_2,p_1),(q_2,q_1)\\
(0,1),(0,1),(p_2,p_1),(q_2,q_1)
\end{array}
\right]
\end{align*}
for $x>0$. 
From (\ref{eqn:tg1}) we have relations for sum:
\begin{align*}
&H^{22}_{44}\left[x\Big|
\begin{array}{cccc}
(\tfrac{\alpha_1-1}{\alpha_1},\tfrac{1}{\alpha_1}),(\tfrac{\alpha_2-1}{\alpha_2},
\tfrac{1}{\alpha_2}),(p_2,p_1),(q_2,q_1)\\
(0,1),(0,1),(p_2,p_1),(q_2,q_1)
\end{array}\right]+\\
&+H^{22}_{44}\left[x\Big|
\begin{array}{cccc}
(\tfrac{\alpha_1-1}{\alpha_1},\tfrac{1}{\alpha_1}),(\tfrac{\alpha_2-1}{\alpha_2},
\tfrac{1}{\alpha_2}),(p_1,p_2),(q_1,q_2)\\
(0,1),(0,1),(p_1,p_2),(q_1,q_2)
\end{array}\right]=\\
&=\alpha_1\alpha_2 g(x;\alpha_1,\alpha_2,p_1,q_1),
\end{align*}
where $g(x)$ is the density given in (\ref{eqn:st_prod1})-(\ref{eqn:st_prod3c}) 
and $x>0$.
For difference of these two particular H-functions, we have:
\begin{align*}
&H^{22}_{44}\left[x\Big|
\begin{array}{cccc}
((\tfrac{\alpha_1-1}{\alpha_1},\tfrac{1}{\alpha_1}),(\tfrac{\alpha_2-1}{\alpha_2},
\tfrac{1}{\alpha_2}),(p_2,p_1),(q_2,q_1)\\
(0,1),(0,1),(p_2,p_1),(q_2,q_1)
\end{array}\right]-\\
&-H^{22}_{44}\left[x\Big|
\begin{array}{cccc}
((\tfrac{\alpha_1-1}{\alpha_1},\tfrac{1}{\alpha_1}),(\tfrac{\alpha_2-1}{\alpha_2},
\tfrac{1}{\alpha_2}),(p_1,p_2),(q_1,q_2)\\
(0,1),(0,1),(p_1,p_2),(q_1,q_2)
\end{array}\right]=\\
&=\tfrac{1}{\pi}\sum_{k=1}^{\infty}
\Gamma\left(\tfrac{k}{\alpha_1}\right)
\Gamma\left(\tfrac{k}{\alpha_2}\right)
\sin([p_1+q_1]k\pi)\frac{x^{k-1}}{[(k-1)!]^2},
\end{align*}
and we assume $\alpha_1,\alpha_2 >1$. 

The other relations can be explored and established by combinations
of results in section three and four.

The product of $n$ independent stable random variables can be computed
by general method in \cite{thompson} combined with method suggested in Proofs 
in section 4. In that case, the integrand in 
(\ref{eqn:st_mb1}) will have poles of $n$th order. These 
can be computed by expanding each term of inegrand in (\ref{eqn:st_mb1})
and it will lead to occurance of polygamma functions in the power-series.
This result can be found useful for example in combination 
with Lagrange's inversion theorem in order to derive power-series 
for quantile function of stable laws.

\end{document}